\documentclass[12pt]{article}

\usepackage{algorithm, algpseudocode}
\usepackage{amsmath,amssymb,amsthm}
\usepackage{mathrsfs}
\usepackage{undertilde} 

\usepackage{color} 
\usepackage{graphicx}
 \vspace{0.6cm}

\newtheorem{lemma}{Lemma}[section]

\newtheorem{them}[lemma]{Theorem}
\newtheorem{lm}[lemma]{Lemma}

\newtheorem{Conjecture}{Conjecture}
\begin{document}

\title{{A simple arithmetic criterion for graphs being determined by their generalized spectra}\footnote{This work is supported by the National Natural Science Foundation of China (No. 11471005)}}

\author{\qquad \small Wei Wang\footnote{The corresponding author. E-mail address:wang$\_$weiw@163.com.} \\
\small School of Mathematics and Statistics, Xi'an Jiaotong University,\\
\small No. 28 Xianning West Rd., Xi'an, Shaanxi, P.R. China, 710049}
\date{}
 \maketitle

\abstract A graph $G$ is said to be determined by its generalized spectrum (DGS for short) if
for any graph $H$, $H$ and $G$ are cospectral with cospectral complements implies that
$H$ is isomorphic to $G$.

It turns out that whether a graph $G$ is DGS is closely related to the arithmetic
properties of its walk-matrix. More precisely, let $A$ be the adjacency matrix of a graph $G$, and let $W =[e, Ae, A^2e,\cdots,A^{n-1}e]$ ($e$ is the all-one vector) be its \textit{walk-matrix}. Denote by $\mathcal{G}_n$ the set
of all graphs on $n$ vertices with $\det(W)\neq 0$. In [Wang, Generalized spectral characterization of graphs revisited, The Electronic J. Combin., 20 (4),(2013), $\#P_4$], the author defined a large family of graphs
$$\mathcal{F}_n = \{G \in{\mathcal{G}_n}|~\frac{\det(W)}{2^{\lfloor\frac{n}{2}\rfloor}}~\mbox{is~square-free~and}~2^{n/2+1}\not|\det(W)\}$$
(which may have positive density among all graphs, as suggested by some numerical
experiments) and conjectured every graph in $\mathcal{F}_n$ is DGS.

 In this paper, we show that the conjecture is actually true, thereby giving a simple arithmetic condition for determining whether a graph is DGS.

\section{Introduction}

The spectra of graphs encodes a lot of combinatorial information about the given graphs, and thus has long been a useful tool in dealing with various problems in Graph Theory, even if they have nothing to do with graph spectra in the appearance.

A fundamental problem in the theory of graph spectra is: ``What kinds of graphs are determined by the spectrum (DS for short)?" The problem dates back to more than 50 years ago and originates from Chemistry, which has received a lot of attention from researchers in recent years.

It was commonly believed that every graph is DS until the first counterexample was found by Collatz and Sinogowitz \cite{CS} in 1957. Since then, various constructions of cospectral graphs (i.e., graphs having the same spectrum) have been studied extensively and a lot of results are presented in literature. For example, Godsil and McKay \cite{GM} invented a powerful method call \textit{GM-switching}, which can produce lots of pairs of cospectral graphs (with cospectral complements). An even more striking result was given by Schwenk~\cite{Sch}, stating that almost all trees are not DS.

However, less results are known about DS graphs, and it turns out that proving graphs to be DS is much more difficult than constructing cospectral graphs. Up to now, all the known DS graphs have very special properties, and the techniques (e.g., the eigenvalue interlacing technique) involved in proving them to be DS depend heavily on some special properties of the spectra of these graphs, and cannot be applied to general graphs. For the background and some known results about this problem, we refer the reader to \cite{DH,DH1} and the references therein.

The above problem clearly depends on the spectrum concerned. In \cite{WX,WX1}, Wang and Xu gave a method for determining whether a graph $G$ is determined
by its generalized spectrum (DGS for short, see Section 2 for details), which works for a large family of general graphs. The key observation is as follows:

Let $G$ and $H$ be two graphs that are cospectral with cospectral complements. Then there exists an orthogonal matrix $Q$ with $Qe=e$ ($e$ is the all-one matrix) such that $Q^TA(G)Q=A(H)$, where $A(G)$ and $A(H)$ are the adjacency matrices of $G$ and $H$, respectively. Moreover, the $Q$ can be chosen to be a rational matrix (under mild restrictions). Thus, if we can show that every rational orthogonal matrix $Q$ with $Qe=e$ such that $Q^TA(G)Q$　is a $(0,1)$-matrix with zero diagonal must be a permutation matrix, then $G$ is clearly DGS. This seems, at first glance, as difficult as the original problem. However, the authors managed to find some algorithmic methods to achieve this goal, by using some arithmetic properties of the walk-matrix associated with the given graph.

In Wang~\cite{WWW}, the author continued this line of research by showing that the DGS-property of a graph $G$ is actually closely related to whether the determinant of the walk-matrix $\det(W)$ is square-free~(for odd primes). More precisely, the author defined a large family of graphs $\mathcal{F}_n$ (see Section 2 for details) that consists of graphs $G$ with $\frac{\det(W)}{2^{\lfloor\frac{n}{2}\rfloor}}$~(this is always an integer; see Section 3) being an odd square-free integer. Then he was able to show that for any graph $G\in{\mathcal{F}_n}$, if $Q$ is a rational orthogonal matrices $Q$ with $Qe=e$ such that $Q^TA(G)Q$　~is a $(0,1)$-matrix with zero diagonal, then $2Q$ must be an integral matrix, and further proposed the following conjecture:

\begin{Conjecture} [Wang~\cite{Wang,WWW}]Every graph in $\mathcal{F}_n$ is DGS.
\end{Conjecture}

The main objective of this paper is to show that the above conjecture is actually true. Thus we have the following theorem.
\begin{them}\label{Main} Conjecture 1 is true.
\end{them}

The proof of above theorem is based on our previous work in \cite{WX,WWW}, and a new insight in dealing with the case $p=2$.

The paper is organized as follows: The next section, we review some previous results
that will be needed in the sequel. In Section 3, we present the proof of Theorem \ref{Main}. In Section 4, we give an extension of the Theorem 1.1. Conclusions and future work are given in Section 5.

\section{Preliminaries}
For convenience of the reader, in this section, we will briefly review some known results
from \cite{WX,WWW}.

Throughout, let $G=(V,E)$ be a simple graph with $(0,1)$-adjacency matrix $A=A(G)$.
The \textit{spectrum} of $G$ consists of all the eigenvalues
(together with their multiplicities) of the matrix $A(G)$. The
spectrum of $G$ together with that of its complement will be
referred to as \textit{the generalized spectrum} of $G$ in the
paper (for some notions and terminologies in graph spectra, see~\cite{CDS}).

For a given graph $G$, we say that $G$ is \textit{determined by
its spectrum} (DS for short), if any graph having the same
spectrum as $G$ is necessarily isomorphic to $G$. (Of course, the
spectrum concerned should be specified.)

The \textit{walk-matrix} of a graph $G$, denoted by $W(G)$ or simply $W$, is defined as $[e, Ae, A^2e,\cdots,A^{n-1}e]$ ($e$ denotes the all-one vector henceforth).
There is a well-known combinatorial interpretation of $W$, that is, the $(i,j)$-th entry of $W$ is the number of walks of $G$ starting from vertex $i$ with
length $j-1$. It
turns out that the arithmetic properties of $\det(W)$ is closely related to wether $G$ is DGS
or not, as we shall see later.

A graph $G$ is called \textit{controllable graph} if $W$ is non-singular (see also \cite{G}). Denote by $\mathcal{G}_n$ the set of all controllable graphs on $n$ vertices.
The following theorem lies at the heart of our discussions.
\begin{them}[\cite{WX}] Let $G\in{\mathcal{G}_n}$. Then there exists a graph $H$ that is
cospectral with $G$ w.r.t. the generalized spectrum if and only if there exists a rational
orthogonal matrix $Q$ such that $Q^TA(G)Q = A(H)$ and $Qe = e$.
\end{them}

Define
$$\mathcal{Q}_G=\left\{ \begin{array}{rrr}Q~\mbox{is~a~rational~orthogonal}&\vline&Q^TAQ ~\mbox{is~a~ symmetric~(0,1)-matrix}\\
   ~\mbox{matrix with}~Qe=e~~~~&\vline&\mbox{with~zero~diagonal~~~~~~~~~~}
  \end{array}\right\},$$ where $e$ is the all-one vector. We have the following theorem:

\begin{them}[\cite{WX}]\label{XXX} Let $G\in\mathcal{G}_n$. Then $G$ is DS w.r.t. the generalized spectrum if and only if the set $\mathcal{Q}_G$ contains only permutation matrices.
\end{them}

By the theorem above, in order to determine whether a given graph
$G\in\mathcal{G}_n$ is DGS or not w.r.t. the generalized spectrum,
one needs to determine all $Q$'s in $\mathcal{Q}_G$ explicitly.
At first glance, this seems to be as difficult as the original problem.
However, we have managed to overcome this difficulty by introducing the following useful notion.

The \textit{level} of a rational orthogonal matrix $Q$ with $Qe=e$
is the smallest positive integer $\ell$ such that $\ell Q$ is an integral
matrix. Clearly, $\ell$ is the least common denominator of all the
entries of the matrix $Q$. If $\ell = 1$, then clearly $Q$ is a permutation matrix.

 Recall that an $n$ by $n$
matrix $U$ with integer entries is called \textit{unimodular} if $\det(U) = \pm1$. The \textit{Smith Normal
Form }(SNF in short) of an integral matrix $M$ is of the form $diag(d_1,d_2,\cdots,d_n)$, where
$d_i$ is the $i$-th elementary divisor of the matrix $M$ and $d_i|d_{i+1}~(i = 1,2,\cdots,n-1)$ hold.
The following theorem is well known.

\begin{them} For every integral matrix $M$ with full rank, there exist unimodular
matrices $U$ and $V$ such that $ M = USV = Udiag(d_1,d_2,\cdots,d_n)V$ , where $S$ is the SNF
of the matrix $M$.
\end{them}

The following theorem shows that the level a rational orthogonal matrix $Q\in{\mathcal{Q}(G)}$ always divides the $n$-th elementary divisor of the walk-matrix.
\begin{them}[\cite{WX}]\label{L0}
Let $W$ be the walk-matrix of a graph $G\in{\mathcal{G}_n}$,
and $Q\in{\mathcal{Q}(G)}$ with level $\ell$. Then we have $\ell|d_n$, where $d_n$ is the $n$-th elementary
divisor of the walk-matrix $W$.
\end{them}
By the above theorem, $\ell$ is a divisor of $d_n$, and hence
is a divisor of $\det(W)$. However, not all divisors of $\det(W)$ can be a divisor of $\ell$, as shown by the following theorem.

\begin{them}[\cite{WWW}] \label{FF} Let $G\in{\mathcal{G}_n}$. Let $Q \in{\mathcal{ Q}_G}$ with level $\ell$, and $p$ be an odd prime. If $p|\det(W)$ and $p^2\not|\det(W)$, then $p$ cannot be a divisor of $\ell$.
\end{them}

Motivated by above theorem, in \cite{WWW}, the author introduced a large family of graphs~(which might have density around 0.2, as suggested by some numerical experiments; see Section 4):
\begin{equation}\label{Defi}
\mathcal{F}_n = \{G \in{\mathcal{G}_n}|~\frac{\det(W)}{2^{\lfloor\frac{n}{2}\rfloor}}~\mbox{is~an odd~square-free~integer}\}.
\end{equation}

As a simple consequence of Theorem~\ref{FF}, we have

\begin{them}\label{Old}
Let $G\in{\mathcal{F}_n}$. Let $Q\in{\mathcal{Q}_G}$ with level $\ell$. Then either $\ell=2^m$ for some integer $m\geq 0$.
\end{them}
Thus, if we can eliminate the possibility that $2\not|\ell$, then Theorem~\ref{Main} follows immediately. In the next section, we will show this is actually the true, which gives a proof of Theorem~\ref{Main}.

\section{Proof of Theorem 1.1}
In this section, we give the proof of Theorem \ref{Main}. Before doing so, we need several lemmas below, the first few of which are taken from \cite{WWW}. In what follows, we will use the finite $\textbf{F}_p$ and ${\rm mod}~p$ (for a prime $p$) interchangeably.

\begin{lm}[c.f.~\cite{WWW}]\label{NB} Let $G\in{\mathcal{G}_n}$. If there is a rational orthogonal matrix $Q\in{\mathcal{Q}_G}$ with level $\ell$ such that $2|\ell$, then there exists a (0,1)-vector $u$ with $u\not\equiv 0~(\rm{mod}~2)$ such that
\begin{equation}\label{EE1}
u^TA^ku\equiv 0~({\rm mod}~4),~{\rm for}~ k=0,1,2,\cdots,n-1.
\end{equation}
 Moreover, $u$ satisfies $W^Tu\equiv~0~(\rm {mod}~ 2)$.
\end{lm}
\begin{proof} $Q\in{\mathcal{Q}_G}$ implies that $Q^TAQ = B$ for some $(0,1)$-matrix $B$ which is the adjacency matrix of a graph $H$. Let $\bar{u}$ be the $i$-th
column of $\ell Q$ with $\bar{u}\not\equiv 0~(\rm{mod}~2)$ (such a $\bar{u}$ always exists by the definition of the level of $Q$). It follows from $Q^TA^kQ = B^k$ that $\bar{u}^TA^k\bar{u} = \ell^2(B^k)_{i,i}\equiv 0~(\rm mod~4)$.
 Let $\bar{u}= u+ 2v$, where $u$ is a $(0,1)$-vector and $v$ is an integral vector.
 Then
$$\bar{u}^TA^k\bar{u} = u^TA^ku + 4u^TA^kv + 4v^TA^kv\equiv 0~ (\rm~mod~4).$$
Thus, Eq. (\ref{EE1}) follows. To show the last assertion, notice that $Q^TA^kQ = B^k$ and $Qe = e$, it follows that $$Q^T[e,Ae,\cdots,A^{n-1}e]=[e,Be,\cdots,B^{n-1}e],$$
i.e., $W(G)^TQ=W(H)$ is an integral matrix. Thus $W(G)^T u \equiv 0~(\rm~mod~2)$ holds. This completes the proof.
\end{proof}

\begin{lm}[\cite{WWW}] \label{Le1}$e^TA^le$ is even for any integer $l\geq 1$.
\end{lm}
\begin{proof} We give a short proof for completeness. Let $A^l:=(b_{ij})$. Note that
\begin{eqnarray*}e^TA^{l}e &=& {\rm Trace}(A^l) +\sum_{i\neq j}b_{ij}\\
& =& {\rm Trace}(A^l) + 2\sum_{1\leq i<j\leq n}b_{ij}\\
&\equiv& {\rm Trace}(A^l)~(\rm mod~2).
\end{eqnarray*}
Moreover, we have ${\rm Trace}(A^l) = {\rm Trace}(AA^{l-1}) =\sum_{i,j}a_{ij}\tilde{b}_{ij}=2\sum_{i<j}a_{ij}\tilde{b}_{ij}$, where$A^{l-1}:=(\tilde{b}_{ij})$. Thus the lemma follows.
\end{proof}

\begin{lm}[\cite{WWW}]\label{H1} ${\rm rank}_2(W)\leq \lceil\frac{n}{2}\rceil$, where ${\rm rank}_2(W)$ denotes the rank of $W$ over the finite field $\textbf{F}_2$.
\end{lm}

\begin{lm}[\cite{WWW}] Let $\det(W)=\pm 2^\alpha p_1^{\alpha_1}p_2^{\alpha_2}\cdots p_s^{\alpha_s}$ be the standard prime decomposition of $\det(W)$. Then $\alpha\geq \lfloor\frac{n}{2}\rfloor$.
\end{lm}

\begin{lm}\label{PO} Let $G\in{\mathcal{F}_n}$. Then the SNF of $W$ is $$S=diag(\underbrace{1,1,\cdots,1}_{\lceil\frac{n}{2}\rceil},\underbrace{2,2,\cdots,2b}_{\lfloor\frac{n}{2}\rfloor}),$$ where the number of 2 in the diagonal of $S$ is $\lfloor\frac{n}{2}\rfloor$ and $b$ is an odd square-free integer. Moreover, we have $rank_2(W)=\lceil\frac{n}{2}\rceil$.
\end{lm}
\begin{proof} By the definition of $\mathcal{F}_n$, we have $\det(W)=\pm 2^{\lfloor\frac{n}{2}\rfloor}p_1p_2\cdots p_s$, where $p_i$ is an odd prime number for each $i$. Thus the SNF of $W$ can be written as $S=diag(1,1,\cdots,1,2^{l_1},2^{l_2},\cdots, 2^{l_t}b)$, where $b=p_1p_2\cdots p_s$ is an odd square-free integer.
It follows from Lemma~\ref{H1} that $rank_2(W)\leq \lceil\frac{n}{2}\rceil$, i.e., $n-t\leq \lceil\frac{n}{2}\rceil$.
Thus, we have $t\geq n-\lceil\frac{n}{2}\rceil=\lfloor\frac{n}{2}\rfloor$. Moreover, we have $l_1+l_2+\cdots+l_t=\lfloor\frac{n}{2}\rfloor$, since $\det(W)=\pm \det(S)$. It follows that $l_1=l_2=\cdots=l_t=1$ and $t=\lfloor\frac{n}{2}\rfloor$, and $rank_2(W)=n-t=\lceil\frac{n}{2}\rceil$.

\end{proof}

\begin{lm}\label{MN} Let $G\in{\mathcal{G}_n}$ and ${\rm rank}_2(W)=\lceil\frac{n}{2}\rceil$. Then any set of $\lfloor\frac{n}{2}\rfloor $ independent column vectors of $W$ (when $n$ is odd, the first column of $W$ is not included) forms a set of fundamental solutions to $W^Tx\equiv 0~(\rm mod~2)$.
\end{lm}
\begin{proof}We distinguish the following two cases.\\
 \textbf{Case 1}. $n$ is even. Let $W^TW=(w_{ij})_{n\times n}$, where $w_{ij}=e^TA^{i+j-2}e$. It follows from Lemma \ref{Le1} and the fact $n$ is even that $W^TW\equiv ~0~(\rm mod~2)$. Notice that the dimension of the the solution space of $W^Tx\equiv 0~(\rm mod~2)$ is $n-{\rm rank}_2(W)=\frac{n}{2}$. Using the assumption ${\rm rank}_2(W)=\frac{n}{2}$ again, we know that any $\frac{n}{2}$ independent column vectors of $W$ forms a set of fundamental solutions to $W^Tx\equiv 0~(\rm mod~2)$.

 \noindent
\textbf{ Case 2}. $n$ is odd. Let $\hat{W}$ be the matrix obtained from $W$ by deleting its first column. Similar to Case 1, we have $W^T\hat{W}\equiv 0~(\rm mod~2)$. Note the dimension of the solution space of $W^Tx\equiv 0~(\rm mod~2)$ is $n-{\rm rank}_2(W)=n-\frac{n+1}{2}=\frac{n-1}{2}$. Moreover, we have ${\rm rank}_2(\hat{W})\geq {\rm rank_2(W})-1=\frac{n-1}{2}$. Therefore, any $\frac{n-1}{2}$ columns from $\hat{W}$, or equivalently, from $W$ (except the first column), forms a set of fundamental solutions to $W^Tx\equiv 0~(\rm mod~2)$.

Combing Cases 1 and 2, the proof is complete.

\end{proof}

\begin{them}[Sach's coefficients Theorem~\cite{CDS}] \label{SACH} Let $P_G(x)=x^n+c_1x^{n-1}+\cdots+c_{n-1}x+c_n$ be the characteristic polynomial of graph $G$. Then
$$c_i=\sum_{H\in{\mathcal{H}_i}}(-1)^{p(H)}2^{c(H)},$$
 where $\mathcal{H}_i$ the set of elementary graphs with $i$ vertices in $G$; $p(H)$ is the number of components of $H$ and $c(H)$ is the number of cycles in $H$.
\end{them}

\begin{lm}\label{beauty} Let $M$ be an integral symmetric matrix. If $M^2\equiv O~(\rm mod~2)$, then $Me\equiv 0~(\rm mod~2)$.
\end{lm}
\begin{proof} Let $M=(m_{ij})$. Then the $(i,i)$-th entry of $M^2$ is $\sum_{j=1}^{n}m_{ij}^2\equiv \sum_{j=1}^{n}m_{ij}\equiv 0~(\rm mod~2)$, which gives that $Me\equiv 0~(\rm mod~2)$.
\end{proof}

Next, we fix some notations. Set $k=\lceil\frac{n}{2}\rceil$.
  Let $\tilde{W}$ be the matrix defined as follows: if $n$ is even, $\tilde{W}$ consists of the first $k$ columns of $W$, i.e., $\tilde{W}=[e,Ae,\cdots,A^{k-1}e]$ ; if $n$ is odd, $\tilde{W}$ consists of the first $k$ columns of $W$, except the first column, i.e., $\tilde{W}=[Ae,A^2e,\cdots,A^{k-1}e]$.
Let $W_1=[e,A^2e,\cdots,A^{2n-2}e]$. Similarly, $\tilde{W}_1$ is defined as $\tilde{W}_1=[e,A^2e,\cdots,A^{2k-2}]$ if $n$ is even; and $\tilde{W}_1=[A^2e,A^4e,\cdots,A^{2k-2}]$ if $n$ is odd.

\begin{lm}\label{NBA} Using notations above, we have

 (i)~${\rm rank}_2(\tilde{W}_1)={\rm rank}_2(W_1)$; (ii)~${\rm rank}_2(\tilde{W})={\rm rank}_2(W)$.
\end{lm}

\begin{proof} We only prove the case that $n$ is even, the case that $n$ is odd can be proved in a similar way.

(i)  Let $P_G(x)=x^n+c_1x^{n-1}+\cdots+c_{n-1}x+c_n$ be the characteristic polynomial of graph $G$. By Sach's Theorem~\ref{SACH}, $c_i$ is even when $i$ is odd, since the number of cycles
must be larger than or equal to one in an elementary subgraph of $G$ with odd number of vertices.

By Hamilton-Cayley's Theorem, we have
$$A^n+\sum_{i=1}^{n}c_iA^{n-i}\equiv A^n+\sum_{j=1}^{n/2}c_{2j}A^{n-2j}\equiv 0~(\rm mod~2).$$
 It follows that $A^ne$ is the linear combinations of
$e,A^2e,\cdots,A^{n-2}e$. Thus, $A^{n+m}e$ is the linear combinations of
$e,A^2e,\cdots,A^{n-2}e$ for any $m\geq 1$. That is, the last $k$ columns of $W_1$ can be expressed as linear combinations of
the first $k$ columns of $W_1$. So (i) follows.

(ii) By (i), we have Let $A^{n}+c_2A^{n-2}+\cdots+c_{n-2}A^2+c_nI=0$. Let $M=A^{n/2}+c_2A^{(n-2)/2}+\cdots+c_{n-2}A+c_nI$. Then we have
\begin{eqnarray*}M^2&\equiv&(A^{n/2}+c_2A^{(n-2)/2}+\cdots+c_{n-2}A+c_nI)^2\\
&\equiv &A^{n}+c_2^2A^{n-2}+\cdots+c_{n-2}^2A^2+c_n^2I\\
&\equiv &A^{n}+c_2A^{n-2}+\cdots+c_{n-2}A^2+c_nI\\
&\equiv& 0~(\rm mod~2).
\end{eqnarray*}
Then, by Lemma~\ref{beauty}, we have $Me=A^{n/2}e+c_2A^{(n-2)/2}e+\cdots+c_{n-2}Ae+c_ne\equiv 0~(\rm mod~2)$.
That is, $A^{n/2}e$ can be expressed as the linear combinations of the first $k$ columns of $W$, and the same is true for $A^{n/2+m}e$, for any $m\geq 0$. That is, any column of $W$ can be expressed as linear combinations of
the first $k$ columns of $W$. So (ii) follows.

\end{proof}

\begin{lm}\label{core}Let $G\in{\mathcal{F}_n}$.  Then we have ${\rm rank}_2(\frac{W^T\tilde{W}_1}{2})=k$ if $n$ is even; and ${\rm rank}_2(\frac{W^T\tilde{W}_1}{2})=k-1$ if $n$ is odd, where $k=\lceil n/2\rceil$.
\end{lm}
\begin{proof} We distinguish the following two cases:

(i) $n$ is even. Write $W=[\tilde{W}_1,\tilde{W}_2]P$, where $P$ is a permutation matrix and $\tilde{W}_2=[Ae,A^3e,\cdots,A^{n-1}e]$. First we show that ${\rm rank}_2(\frac{W^TW}{2})=n$.
Actually, notice that $G\in{\mathcal{F}_n}$, we have $\det(W)=\pm 2^{n/2}b$, where $b$ is an odd integer.
It follows that $\det(W^TW)= 2^nb^2$, i.e., $\det(\frac{W^TW}{2})=b^2$. Note that $b$ is odd, the assertion follows immediately. Now we have $\frac{W^TW}{2}=[\frac{W^T\tilde{W}_1}{2},\frac{W^T\tilde{W}_2}{2}]P$. It follows that the column vectors of the matrix $\frac{W^T\tilde{W}_1}{2}$ are linearly independent (since $\frac{W^TW}{2}$ has full rank), over $\textbf{F}_2$.

(ii) $n$ is odd. Construct a new matrix $\hat{W}=[2e,\tilde{W}_1,\tilde{W}_2]$. Notice that $\frac{W^T\hat{W}}{2}$ is now always an integral matrix. Since $\det(W)=2^{(n-1)/2}b$ $(b$ is odd), we have $\det(W^T\hat{W})=\det(W)\det(\hat{W})=\pm 2\det^2(W)=2^{n}b^2$ (since $\det(\hat{W})=\pm 2\det(W)$). It follows that $\frac{W^T\hat{W}}{2}=[W^Te,\frac{W^T\tilde{W}_1}{2},\frac{W^T\tilde{W}_2}{2}]$ has full rank $n$.
Therefore, ${\rm rank}_2(\frac{W^T\tilde{W}_1}{2})$ equals the number of columns of $\tilde{W}_1$, which is $k-1$ when $n$ is odd.

Combining Cases (i) and (ii), the lemma is true. The proof is complete.
\end{proof}

The following lemma lies at the heart of the proof of Theorem~\ref{Main}.

\begin{lm}\label{L1} Let $G\in{\mathcal{F}_n}$. Let $Q\in{\mathcal{Q}_G}$ be a rational orthogonal matrix with level $\ell$, then $2\not|\ell$.
\end{lm}

\begin{proof} We prove the lemma by contradiction. Suppose on the contrary, $2|\ell$. It follows from Lemma~\ref{NB} that
there exists a vector $u$ such that Eq. (\ref{EE1}) holds. Note that $u$ is a solution to the system of
linear equations $W^Tx\equiv 0~(\rm mod~2)$.
 Since $G\in{\mathcal{F}_n}$, it follows from Lemma~\ref{PO} that ${\rm rank}_2(W)=\lceil\frac{n}{2}\rceil$.
 According to Lemmas~\ref{MN} and \ref{NBA}, we can assume that $\{A^{i_{1}}e,A^{i_{2}}e,\cdots,A^{i_{k}}\}$ is a set of fundamental solutions to $W^Tx\equiv 0~(\rm mod~2)$, where $k:=\lfloor n/2\rfloor$, and $i_1=0,i_2=1,\cdots,i_k=n/2-1$ if $n$ is even and $i_1=1,i_2=2,\cdots,i_k=(n-1)/2$ if $n$ is odd.

Write $\tilde{W}=[A^{i_1}e,A^{i_2}e,\cdots,A^{i_k}e]$. Then $u$ can be written as the linear combinations of the column vectors of $\tilde{W}$, i.e., there is a vector $v\not \equiv 0~(\rm mod~2)$ such that $u\equiv\tilde{W}v~(\rm mod~2)$. So we have $u=\tilde{W}v+2\beta$ for some integral vector $\beta$. It follows that 
\begin{eqnarray*}u^TA^lu&=&(\tilde{W}v+2\beta)^TA^l(\tilde{W}v+2\beta)\\
&=&v^T\tilde{W}^TA^l\tilde{W}v+2v^T\tilde{W}^TA^l\beta+2\beta^TA^l\tilde{W}v+4\beta^TA^l\beta\\
&=&v^T\tilde{W}^TA^l\tilde{W}v+4v^T\tilde{W}^TA^l\beta+4\beta^TA^l\beta\\
&\equiv&v^T\tilde{W}^TA^l\tilde{W}v~(\rm mod~4).
\end{eqnarray*}

By Eq.~(\ref{EE1}), we have $v^T\tilde{W}^TA^l\tilde{W}v\equiv 0~(\rm mod~4)$, for $l=0,1,2,\cdots,n-1$.
Notice that
\begin{eqnarray*}
\tilde{W}^TA^l\tilde{W}=\left[\begin{array}{cccc}
e^TA^{2i_1+l}e & e^TA^{i_1+i_2+l}e &\cdots & e^TA^{i_1+i_k+l}e\\
e^TA^{i_1+i_2+l}e & e^TA^{2i_2+l}e &\cdots & e^TA^{i_2+i_k+l}e\\
\vdots&\vdots&\ddots&\vdots\\
e^TA^{i_1+i_k+l}e & e^TA^{i_2+i_k+l}e &\cdots & e^TA^{2i_k+l}e
\end{array}\right]
\end{eqnarray*}

Let $M:=\tilde{W}^TA^l\tilde{W}$. A key observation is that $M$ is always a symmetric matrix with every entry being a multiple of two. Actually, this follows from Lemma~\ref{Le1}. But we have to distinguish two cases: (i) when $n$ is even, Lemma~\ref{Le1} always can be applied except the case that $i_1=l=0$. While in this case, the $(1,1)$-entry of $M$ is $e^Te=n$ which is even; (ii) when $n$ is odd, we have $i_1=1$, thus applying Lemma~\ref{Le1} directly leads to the desired assertion.

Let $v=(v_1,v_2,\cdots,v_k)^T$. Then we have
 $$M_{ij}v_iv_j+M_{ji}v_jv_i=2M_{ij}v_iv_j\equiv 0~(\rm mod~4),$$
 for $i\neq j$, since $M_{ij}$ is even by the above discussions. Therefore, we have
\begin{eqnarray*}
v^T\tilde{W}^TA^l\tilde{W}v&=&\sum_{i,j}M_{ij}v_iv_j\\
&\equiv &(e^TA^{2i_1+l}e)v_1^2+(e^TA^{2i_2+l}e)v_2^2+\cdots+(e^TA^{2i_k+l}e)v_k^2\\
&\equiv& (e^TA^{2i_1+l}e)v_1+(e^TA^{2i_2+l}e)v_2+\cdots+(e^TA^{2i_k+l}e)v_k\\
&\equiv& [e^TA^{2i_1+l}e,e^TA^{2i_2+l}e,\cdots,e^TA^{2i_k+l}e]v\\
&\equiv& 0~~(\rm mod~4),
\end{eqnarray*}
for $l=0,1,\cdots,n-1$. The second congruence equation follows since $(e^TA^{2i_j+l}e)v_j^2\equiv (e^TA^{2i_j+l}e)v_j~(\rm mod~4)$ for every $1\leq j\leq k$.

Let $\tilde{W}'$ be an $n$ by $k$ matrix defined as follows:
\begin{eqnarray*}
\tilde{W}'&:=&\left[\begin{array}{cccc}
e^TA^{2i_1}e & e^TA^{2i_2}e &\cdots & e^TA^{2i_k}e\\
e^TA^{2i_1+1}e & e^TA^{2i_2+1}e &\cdots & e^TA^{2i_k+1}e\\
\vdots&\vdots&\ddots&\vdots\\
e^TA^{2i_1+n-1}e & e^TA^{2i_2+n-1}e &\cdots & e^TA^{2i_k+n-1}e
\end{array}\right]\\
&=& \left[\begin{array}{c}
e^T\\
e^TA\\
\vdots\\
e^TA^{n-1}
\end{array}\right]
\left[\begin{array}{cccc}
A^{2i_1}e & A^{2i_2}e & \cdots &A^{2i_k}e
\end{array}\right]\\
&=& W^T\left[\begin{array}{cccc}
A^{2i_1}e & A^{2i_2}e & \cdots &A^{2i_k}e
\end{array}\right]\\
&=&W^T\tilde{W}_1,
 \end{eqnarray*}
 where $\tilde{W}_1=[A^{2i_1}e ,A^{2i_2}e, \cdots,A^{2i_k}e]$.

Thus, we have $W^T\tilde{W}_1v\equiv 0~(\rm mod~4)$. Notice that $\frac{W^T\tilde{W}_1}{2}$ is always an integral matrix according to Lemma~\ref{Le1} and the definition of $\tilde{W}_1$. It follows that
 $$\frac{W^T\tilde{W}_1}{2}v\equiv 0~(\rm mod~2).$$

However, by Lemma~\ref{core}, ${\rm rank}_2(\frac{W^T\tilde{W}_1}{2})=k$ and hence, $\frac{W^T\tilde{W}_1}{2}$ has full column rank. It follows that $v\equiv 0~(\rm mod~2)$; a contradiction. This completes the proof.

\end{proof}

Now, we are ready to present the proof of Theorem \ref{Main}.

\begin{proof} Let $G\in{\mathcal{F}_n}$. Let $Q\in{\mathcal{Q}_G}$ with level $\ell$. Then by Theorem \ref{Old}, we have $p\not|\ell$ for any odd prime $p$. By Lemma~\ref{L1}, we have $2\not|\ell$. It follows that $\ell=1$ and hence, $Q$ is a permutation matrix. By Theorem \ref{XXX}, $G$ is DGS. The proof is complete.
\end{proof}

\section{An extension beyond Theorem 1.1}

In the previous section, we have shown that graphs with $\frac{\det{W}}{2^{\lfloor\frac{n}{2}\rfloor}}$ being square-free is always DGS. Notice graphs with above property has the following SNF: $$S=diag(\underbrace{1,1,\cdots,1}_{\lceil\frac{n}{2}\rceil},\underbrace{2,2,\cdots,2b}_{\lfloor\frac{n}{2}\rfloor}),$$ where $b$ is an odd square-free integer. A natural question is: Can we enlarge the family of graph $\mathcal{F}_n$?

Generally, we cannot expect an affirmative answer to this question if we allow $b$ is not square-free. In \cite{WWW}, the author have given an example of non-DGS graph of order 12 with $\det(W)=2^6\times 3^2\times 157\times1361\times 2237$, which shows Theorem 1.1 is best possible in the sense that we cannot guarantee that $G$ is DGS if $\frac{\det{W}}{2^{\lfloor\frac{n}{2}\rfloor}}$ has prime divisor with exponent larger than one.

However, based on the proof in Lemma~\ref{L1}, we are able to give a method to determine DGS-property for graphs
that are not in $\mathcal{F}_n$. Next, we try to give a method for determine the DGS-property for graphs whose walk-matrices have the following SNF:
\begin{equation}\label{SNF}
diag(\underbrace{1,1,\cdots,1}_{\lceil\frac{n}{2}\rceil},\underbrace{2^{l_1},2^{l_2},\cdots,2^{l_t}b}_{\lfloor\frac{n}{2}\rfloor})
\end{equation}

\begin{lm}\label{LK1} Let $G\in{\mathcal{G}_n}$. Suppose that ${\rm rank}_2(W)=\lceil\frac{n}{2}\rceil$ and the SNF of $W$ is given as in Eq. (\ref{SNF}), where $b$ is a square-free integer. Let $Q\in{\mathcal{Q}_G}$ be a rational orthogonal matrix with level $\ell$. Let $W_1:=[e,A^2e,A^4e,\cdots,A^{2n-2}e]$.
 If \begin{equation}
 \{x|\frac{W^TW_1}{2}x\equiv~0~({\rm mod}~2)\}\subset \{x|Wx\equiv 0~({\rm mod}~2)\},
 \end{equation}
then $2\not|\ell$.
\end{lm}

\begin{proof} The proof is similar to that of Lemma~\ref{L1}. A sketch.
 
  Suppose on the contrary $2|\ell$. It follows from Lemma~\ref{NB} that
there exists a vector $u$ such that Eq. (\ref{EE1}) holds. Note that $u$ is a solution to the system of
linear equations $W^Tx\equiv 0~(\rm mod~2)$. According to Lemma~\ref{MN}, any solution of $W^Tx\equiv 0~(\rm mod~2)$, can be written as linear combinations of the column vectors of $W$ (when $n$ is odd, replace $W$ with $\hat{W}$). It follows that $u$ can be written as the linear combinations of the column vectors of $W$, i.e., there is a vector $v\not \equiv 0~(\rm mod~2)$ such that $u\equiv Wv~(\rm mod~2)$.

Using the similar arguments as in the remaining proof of Lemma~\ref{L1}, we have
 $W^TW_1v\equiv 0~(\rm mod~4)$. Notice that $W^TW_1\equiv 0~(\rm mod~2)$. We have $\frac{W^TW_1}{2}v\equiv 0~(\rm mod~2)$, which implies that $u=Wv\equiv0~(\rm mod~2)$ by the assumption of the lemma; a contradiction.
Therefore $2\not|\ell$. This completes the proof.

\end{proof}

Combining the above lemma and Theorem 2.5, we have the following theorem.
\begin{them} Let $G\in{\mathcal{G}_n}$. Suppose that ${\rm rank}_2(W)=\lceil\frac{n}{2}\rceil$ and the SNF of $W$ is given as in Eq. (\ref{SNF}), where $b$ is a square-free integer. Then $G$ is DGS.
\end{them}

We give an example as an illustration. Let the adjacency matrix of graph $G$ be given as follows:
  
  {\small $$A=\left[\begin{array}{cccccccccccccccccccc}
  {0, 1, 1, 1, 1, 1, 0, 1, 0, 0, 1, 1, 1, 1, 0, 0, 1, 1, 1, 1}\\
   {1, 0, 0, 1, 1, 1, 0, 0, 0, 0, 0, 0, 0, 1, 1, 0, 1, 0, 0, 1}\\
    { 1, 0, 0, 0, 0, 1, 1, 1, 0, 0, 0, 1, 1, 1, 0, 0, 0, 1, 0, 0}\\
     {1, 1, 0, 0, 1, 0, 1, 0, 1, 1, 0, 1, 0, 1, 0, 0, 0, 1, 0, 1}\\
      { 1, 1, 0, 1, 0, 1, 0, 1, 1, 1, 1, 0, 1, 0, 0, 0, 1, 0, 0, 1}\\
       {1, 1, 1, 0, 1, 0, 0, 1, 1, 0, 0, 0, 1, 0, 1, 0, 1, 0, 0, 1}\\
        { 0, 0, 1, 1, 0, 0, 0, 0, 0, 1, 1, 0, 1, 1, 1, 0, 1, 1, 1, 1}\\
         {1, 0, 1, 0, 1, 1, 0, 0, 0, 0, 0, 0, 0, 0, 1, 1, 1, 0, 1, 1}\\
          {0, 0, 0, 1, 1, 1, 0, 0, 0, 0, 1, 1, 1, 0, 0, 0, 0, 1, 1, 0}\\
           {0, 0, 0, 1, 1, 0, 1, 0, 0, 0, 1, 1, 0, 1, 0, 0, 0, 1, 0, 0}\\
            {1, 0, 0, 0, 1, 0, 1, 0, 1, 1, 0, 0, 0, 1, 1, 1, 0, 0, 1, 0}\\
             {1, 0, 1, 1, 0, 0, 0, 0, 1, 1, 0, 0, 1, 1, 0, 0, 0, 0, 0, 0}\\
              {1, 0, 1, 0, 1, 1, 1, 0, 1, 0, 0, 1, 0, 0, 1, 1, 0, 1, 1, 0}\\
               {1, 1, 1, 1, 0, 0, 1, 0, 0, 1, 1, 1, 0, 0, 0, 1, 1, 0, 1, 0}\\
                { 0, 1, 0, 0, 0, 1, 1, 1, 0, 0, 1, 0, 1, 0, 0, 1, 0, 0, 1, 0}\\
                 {0, 0, 0, 0, 0, 0, 0, 1, 0, 0, 1, 0, 1, 1, 1, 0, 0, 1, 1, 0}\\
                  {1, 1, 0, 0, 1, 1, 1, 1, 0, 0, 0, 0, 0, 1, 0, 0, 0, 1, 0, 1}\\
                   {1, 0, 1, 1, 0, 0, 1, 0, 1, 1, 0, 0, 1, 0, 0, 1, 1, 0, 1, 0}\\
                    {1, 0, 0, 0, 0, 0, 1, 1, 1, 0, 1, 0, 1, 1, 1, 1, 0, 1, 0, 1}\\
                     {1, 1, 0, 1, 1, 1, 1, 1, 0, 0, 0, 0, 0, 0, 0, 0, 1, 0, 1, 0}
\end{array}\right].$$}
  It can easily be computed by using Mathematica 5.0 that
$\det(W)=-2^{13}b$, where $b=7\times 11\times 383\times 210857\times
  231734663160530708115251000501057$.
  The SNF of $W$ is as follows:
  $$diag(\underbrace{1,1,\cdots,1}_{10}\underbrace{2,2,\cdots,2,2^2,2^2,2^2}_{10}b).$$
  Moreover, it can be verified that Eq.~(4) also holds. Thus, $G$ is DGS according to Theorem~4.2.

\section{Conclusions and future work}

In this paper, we have given a simple arithmetic criterion for determining whether a graph $G$ is DGS,
in terms of whether the determinant of walk-matrix $\det(W)$ divided by $2^{\lfloor\frac{n}{2}\rfloor}$ is an odd square-free. It is noticed that the definition of $\mathcal{F}_n$ is so simple that the membership of a graph can easily be checked.

We have performed a series of numerical experiments to see how large the family of graphs $\mathcal{F}_n$ is. The graphs are generated randomly independently from the probability space $\mathcal{G}(n,\frac{1}{2})$~(see e.g. \cite{B}). At each time, we generated 1,000 graphs randomly, and counted the number of graphs that are in $\mathcal{F}_n$. Table 1 records one of such experiments (note the results may be varied slightly at each run of the algorithm). The first column is the order $n$ of the graphs generated varying from 10 to 50. The second column records the number of graphs that are belonged to $\mathcal{F}_n$ among the randomly generated 1,000 graphs, and the third column is the corresponding fractions.

 \vspace{3mm}
\renewcommand\arraystretch{0.90}
\begin{center}{\footnotesize
\begin{tabular}{c|c|c}
\multicolumn{3}{l}{\bf Table 1  \quad Fractions of Graphs in $\mathcal{F}_n$}\\
\hline
  $n$ &  \# Graphs in $\mathcal{F}_n$ & The Fractions \\ \hline
  10 &  211 & 0.211 \\ \hline
   15  &  201 & 0.201 \\ \hline
  20  & 213 &0.213 \\ \hline
   25   & 216 &0.216\\ \hline
   30   & 233 &0.233\\ \hline
   35  &  229 &0.229\\ \hline
   40  &  198 &0.198 \\ \hline
   45  &  202  &0.202\\ \hline
   50  &  204 &0.204  \\ \hline

\end{tabular}
}
\end{center}
\vspace{3mm}

We can see from Table 1 that graphs in $\mathcal{F}_n$ has a density around 0.2. It would be an interesting future work to show that this is actually the case.

\end{document}